\documentclass[10pt,reqno,twoside]{siamltex}  
\usepackage{amsmath, amssymb}
\newtheorem{nota}[theorem]{Remark}
\newcommand{\email}[1]{\textsc{e-mail: }\texttt{#1}}
\newcommand{\racion}[2]{\mbox{\small$\frac{{#1}}{{#2}}$} }
\newcommand{\refe}[1]{(\ref{#1})}
\newcommand{\half}{\mbox{$\frac{1}{2}$}}
\newcommand{\dst}{\displaystyle}
\newcommand{\tres}{\mbox{$\frac{3}{2}$ }}
\newcommand{\Ry}{\Rightarrow}
\newcommand{\fhyp}{\mbox{F}}
\newcommand{\q}{\varsigma}
\newcommand{\pe}[2]{\langle#1,#2\rangle}
\newcommand\RR{\mathbb{R}}
\newcommand\CC{\mathbb{C}}
\renewcommand{\a}{\g a^{\downarrow}_0}
\renewcommand{\b}{\g a^{\uparrow}_0}
\newfont{\got}{eufm10 scaled \magstep1}
\newcommand{\g}[1]{\mbox{\got #1}}

\title{Factorization of the hypergeometric-type \\
difference equation on the uniform lattice}
\author{R. Alvarez-Nodarse%
\thanks{Departamento de An\'alisis Matem\'atico,
Universidad de Sevilla, Apdo. 1160. 41080-Sevilla,
Spain and
Instituto Carlos I de F\'{\i}sica Te\'orica y Computacional,
Universidad de Granada, E-18071 Granada, Spain.
\email{ran@us.es}} \and N. M. Atakishiyev%
\thanks{Instituto de Matem\'aticas, UNAM,  Apartado Postal 273-3,
C.P. 62210 Cuernavaca, Morelos, M\'exico.
\email{natig@matcuer.unam.mx}} \and R. S. Costas-Santos%
\thanks{Departamento de Matem\'aticas,
E.P.S., Universidad Carlos III de Madrid.
Ave. Universidad 30, E-28911, Legan\'es, Madrid,
Spain. \email{rcostas@math.uc3m.es}}}

\begin{document}
\maketitle
\pagestyle{myheadings}
\markboth{FACTORIZATION OF DIFFERENCE EQUATIONS}
{R. ALVAREZ-NODARSE, N.M. ATAKISHIYEV, R. COSTAS-SANTOS}
\begin{abstract}%
We discuss factorization of the hy\-per\-geo\-me\-tric-ty\-pe difference
equations on the uniform lattices and show how one can construct
a dynamical algebra, which corresponds to each of these equations.
Some examples are exhibited, in particular, we show that several
models of discrete harmonic oscillators, previously considered
in a number of publications, can be treated in a unified form.
\end{abstract}
\section{Introduction}
The study of discrete system has attracted the attention of many
authors in the last years. Of special interest are the discrete
analogs of the quantum harmonic oscillators \cite{asuslov1,asuslov2,
ata98,asuslov4,asuslov5,ruf1,vin1,vin2,ruf2,nag95} among others.

There are several methods for studying such systems. One of them
is the factorization method (FM), first introduced for solving
differential equations \cite{IH}. This classical FM is based on
the existence of the so-called raising and lowering  operators for
the corresponding equation, which allow to find the explicit
solutions in a simple way, see e.g. \cite{ata84,lor01}. Later on,  Miller
extended it to difference equations \cite{mil69}
and $q$-differences --in the Hahn sense-- \cite{mil70}.
In the case of difference equations this method has been also extensively
used during the last years (see e.g. \cite{ata98,asuslov4,ban98,lor01,smi99}
for difference analogs on the uniform lattice and \cite{ran-rob,asuslov1,asuslov2,ata98,
asuslov4,asuslov5,ban99} for the $q$-case).

Later on, references \cite{ata84,ataw94,ata94} indicated a way of
constructing the so-called ``dynamical symmetry algebra'' by applying the FM to
differential or difference equations \cite{ran03,asuslov4,asuslov5}
and then this technique has been used to consider some particular
instances of $q$-hypergeometric difference equations. Of special
interest is also the paper by Smirnov \cite{smi99}, in which the
equivalence of the FM and the Nikiforov {\it et al} formulation of
theory of $q$-orthogonal polynomials \cite{nsu}, was established.
In \cite{ran-rob}, following the papers \cite{ban99,lor01} for the classical case,
it has been shown that one can factorize the hyper\-geo\-me\-tric-ty\-pe
difference equation \refe{eqdif-q} in terms of the above-mentioned raising
and lowering operators.

Our main purpose here is to show how to deal with all different
cases of difference equations on the uniform lattice $x(s)=s$ in
an unified form. One should consider this paper as an attempt to
provide a background for the more general $q$-linear case (since
in the limit as $q$ goes to $1$, the $q$-linear case reduces to
the uniform one). Some results concerning this general case will be
also given in the last section.

The structure of the paper is as follows. In Section 2 some
necessary results on classical polynomials are collected. In section
\ref{sec-3} the factorization of the hy\-per\-geo\-me\-tric-ty\-pe
difference equation is discussed, which is used in section
\ref{sec-4} to construct a dynamical symmetry algebra in the case
of the Charlier polynomials. In section \ref{sec-5} the Kravchuk and
the Meixner cases are considered in detail. Finally, in section
\ref{sec-6} we briefly discuss a possibility of applying this
technique to the $q$-case.
\section{Preliminaries: the classical ``discrete'' polynomials}
\label{pre}
The discretization of the hypergeometric differential equation
on the lattice $x(s)$ \cite{nsu,nubook} leads to the {\it second
order difference equation of the hypergeometric type}
\begin{equation}
\begin{array}{c} \displaystyle
\sigma(s)\frac{\Delta}{\Delta x(s-1/2)}
\frac{\nabla y(x(s))}{\nabla x(s)}+
\tau(s)\frac{\Delta y(x(s))}{\Delta x(s)}+
\lambda y(x(s)) =0,
\end{array}
\label{eqdif-q}
\end{equation}
where $\Delta f(s):= f(s+1)-f(s)$, $\nabla f(s):= f(s)-f(s-1)$.

The most simple lattice is the \textit{uniform} one $x(s)=s$ and it
corresponds to the equation
\begin{equation}
\sigma(x)\Delta\nabla y(x)+\tau(x)\Delta y(x)+\lambda y(x)=0.
\label{eqdif-d}
\end{equation}
The above equation have polynomial solutions $P_n(s)$, usually
called classical \textit{discrete} orthogonal polynomials,
if and only if $\lambda=\lambda_n=-n(\tau'+(n-1)\sigma''/2)$.

It is well known \cite{nsu} that under certain conditions the
polynomial solutions of \refe{eqdif-d} are orthogonal. For example,
if $\sigma(s)\rho(s)x^{k}\Big|_{s=a,b}=0$, for all $k=0,1,2,\dots$,
then the polynomial solutions $P_n(x)$ of \refe{eqdif-d} satisfy
\begin{equation}\label{dis-ort}
\pe{P_n}{P_m}_d=
\sum_{s=a}^{b-1}P_n(s;q)\, P_m(s;q)\, \rho(s) =\delta_{nm} d_n^2,
\end{equation}
where the weight functions $\rho(s)$ are solutions of the
Pearson-type equation
\begin{equation}\label{pearson}
\Delta\left[\sigma(s) \rho(s)\right]= \tau(s) \rho(s)\quad
\mbox{or} \quad \sigma(s+1)\rho(s+1)=[\sigma(s)+\tau(s)]\rho(s).
\end{equation}
In the following we will consider the monic polynomials, i.e.,
$P_n(x)=x^n+b_n x^{n-1}\!+\!\!~\cdots$.

The polynomial solutions of \refe{eqdif-d} are the classical
\textit{discrete} orthogonal polynomials of Hahn, Meixner, Kravchuk
and Charlier and their principal data are given in the table
\ref{tabla-pol-dis}.

\begin{table}[h]
\caption{The classical \textit{discrete} orthogonal monic polynomials.}
\label{tabla-pol-dis}
\begin{center}
\renewcommand{\arraystretch}{.5}
{\small \hspace{-.5cm}
\begin{tabular}{|c||c|c|c|c|}
\hline
\ & \ & \ & \ &  \\
  \ & Hahn & Meixner  & Kravchuk &  Charlier \\
\ & \ & \ & \ &  \\
$P_n(x)$ & $h_n^{\alpha,\beta}(x;N)$ & $M_n^{\gamma,\mu}(x)$ &
$K_n^{p }(x)$ & $C_n^{ \mu}(x)$  \\
\ & \ & \ & \ &  \\
\hline \hline
 \ & \ & \ & \ &  \\
$\, [a,b] \, $ & $[0,N]$ & $[0,\infty)$ & $[0,N+1]$ & $[0,\infty)$ \\
 \ & \ & \ & \ & \\
\hline
\ & \ & \ & \ & \\
 $\sigma(x) $ &  $x(N+\alpha -x)$ & $x$  & $x$ & $x$ \\
\ & \ & \ & \ &  \\
\hline
\ & \ & \ & \ &  \\
$\tau(x)$ & $(\beta+1)(N-1)-(\alpha+\beta+2)x$  &
$(\mu-1)x+\mu\gamma$  &
 ${\frac{Np-x}{1-p}} $ & $\mu-x$ \\
\ & \ & \ & \ & \\
 \hline
 \ & \ & \ & &  \\
$\, \sigma+\tau$ & $(x+\beta+1)(N-1-x)$ &  $\mu x +\gamma\mu$ &
$ -{\frac{p}{1-p}}(x-N) $ & $\mu$ \\
 \ & \ & \ & & \\
\hline
\ & \ & \ & & \\
  $\lambda_n$& $ n(n+\alpha+\beta+1) $ & $(1-\mu)n $ &
$\frac{n}{1-p}$ & $n$ \\
\ & \ & \ & &\\
\hline
\ & \ & \ & \ &  \\
 $\rho(x)$& $\frac{\Gamma(N+\alpha-x)\Gamma(\beta+x+1) }
{\Gamma(N-x)\Gamma(x+1)}$ &
$\frac{\mu^x  \Gamma(\gamma+x)}{\Gamma(\gamma)\Gamma(x+1)}$  &
$ {N \choose x}p^x(1-p)^{N-x}$
&  $  \frac{e^{-\mu} \mu^x}{\Gamma(x+1)}$ \\
\ & \ & \ & \ &  \\
& \footnotesize{$\alpha,\beta \geq\!-\!1,\, n\leq N\!-\!1$} &
\footnotesize{$\gamma\!>0,\mu\in(0,1)$}  &
\footnotesize{$p\in(0,1),n\leq \!N\!-\!1$} &
\footnotesize{$\mu >0$}  \\
\ & \ & \ & \ & \\\hline
\ & \ & \ & &\\
$d_n^2$ &
$\frac{n!\Gamma(\alpha+\beta+N +n+1)}{(N-n-1)!(\alpha+\beta+n+1)_n}
{\alpha+\beta+2n+2 \choose \alpha+n+1}^{\!\!-\!1}$ &
$   \frac{n!(\gamma)_n \mu^n}{(1-\mu)^{\gamma+2n}}$  &
$ {N \choose n} p^n (1-p)^n$  & $n!\mu^n$   \\
\ & \ & \ & &\\
\hline
 \end{tabular}}
\end{center}
 \end{table}

They can be expressed in terms of the generalized hypergeometric
function ${}_p\fhyp _q$,
\begin{equation}
\begin{array}{l}
\displaystyle\hspace{0.2cm}_p
\fhyp _q\displaystyle
\bigg(\begin{array}{c}{a_1,a_2,\dots,a_p} \\
{b_1,b_2,\dots,b_q}\end{array}
\bigg|x\bigg)=
\displaystyle \sum _{k=0}^{\infty}\frac{ (a_1)_k(a_2)_k
\cdots (a_p)_k} {(b_1)_k(b_2)_k \cdots  (b_q)_k}
\frac{x^k}{k!} ,
\end{array}
\label{def-f-h}
\end{equation}
where $(a)_k$ is  the Pochhammer symbol (or shifted factorial)
\begin{equation}
(a)_0=1 ,\hspace{.25cm} (a)_k=a(a+1)(a+2) \cdots(a+k-1),\,\,
k=1,2,3,\dots \,\,\, .
\label{sim-poc}
\end{equation}
Using the above notations, we have for the monic polynomials of
Hahn, Meixner, Kravchuk and Charlier, respectively
\begin{equation}\begin{array}{c}
\hspace{-.3cm}h_n^{\alpha,\beta}(x,N)=
\dst  \frac{  (1-N)_n (\beta +1)_n }{ (\alpha+\beta+n +1)_n}
 {\dst \,_3\fhyp_2  \displaystyle
\bigg(\begin{array}{c}{-x, \alpha+\beta+n+1,-n}\\
{ 1-N, \beta+1}\end{array} \bigg| 1 \bigg),}
\end{array}\label{hyph}
\end{equation}
\begin{equation}
M_n^{\gamma,\mu}(x)=
\displaystyle \frac{ (\gamma)_{n}\mu^n}{(\mu-1)^n}
 \displaystyle\hspace{0.3cm}_2
\fhyp _1\displaystyle\bigg(\begin{array}{c}{-n,-x}\\{ \gamma}
\end{array} \bigg|1-\frac{1}{\mu}\bigg),
\label{hypm}
\end{equation}
\begin{equation}
K_n^{p}(x) =
\displaystyle  \frac{(-p)^n N!}{(N-n)!}
 \displaystyle\hspace{0.3cm}_2
\fhyp _1\displaystyle\bigg(\begin{array}{c} {-n,-x}\\ { -N}
\end{array} \bigg|
\frac{1}{p}\bigg),
\label{hypk}
\end{equation}
\begin{equation}
C_n^{\mu}(x)=
\displaystyle (-\mu)^n
 \displaystyle\hspace{0.3cm}_2
\fhyp _0\displaystyle
\bigg(\begin{array}{c} {-n,-x}\\{-}\end{array}\bigg|-
\frac{1}{\mu}\bigg).
\label{hypc}
\end{equation}

A further information on orthogonal polynomials on the uniform
lattice can be found in \cite{ran,ks,nsu,nubook}.

\section{Factorization of the difference equation}
\label{sec-3}
Let us consider the following second order linear difference
operator
\begin{equation} \label{ham1}
\g h_1(s)=-{\nu(s-1)}\,e^{-\partial_s}-
{\nu(s)}\,e^{\partial_s}+[2\sigma(s)+\tau(s)]I,
\end{equation}
where $e^{\alpha\partial s}f(s)=f(s+\alpha)$ for all $\alpha\in\CC$,
$\nu(s)=\sqrt{\sigma(s+1)[\sigma(s)+\tau(s)]}$, and $I$ is the
identity operator, and let  $(\Phi_n)_n$  be the set of functions
\begin{equation}
\Phi_n(s)= \frac{\sqrt{\rho(s)}}{d_n} \, P_n(s),
\label{nor-fun}
\end{equation}
where $d_n$ is a norm of the polynomials $P_n(s)$, which satisfy
equation \refe{eqdif-d}, and $\rho(s)$ is the solution of the
Pearson-type equation \refe{pearson}. If $P_n(s)$ possess the
discrete orthogonality property \refe{dis-ort}, then the
functions $\Phi_n(s)$ have the property
\begin{equation}
\pe{\Phi_n(s)}{\Phi_m(s)}_d=\sum_{s=a}^{b-1}{\Phi_n(s)}{\Phi_m(s)}
=\delta_{n,m}.
\label{ort-Phi}
\end{equation}

Using the identity $\nabla=\Delta-\nabla \Delta$ and the
equation \refe{eqdif-d}, one finds that
\begin{equation} \label{gen_ham2}
\g h_1(s)\Phi_n(s) =\lambda_n \Phi_n(s),
\end{equation}
i.e., the functions $\Phi_n(s)$, defined in \refe{nor-fun}, are the
eigenfunctions of $\g h_1(s)$.
In the following we will refer to $\g h_1(s)$ as
the {\em hamiltonian}.

Our first step is to find two operators $a(s)$ and $b(s)$ such that
the Hamiltonian $\g h_1(s)=b(s)a(s)$, i.e., the operators  $a(s)$
and $b(s)$ {\em factorize} the Hamiltonian $\g h_1(s)$.
\begin{definition}
Let $\alpha$ be a real number. We define a family of $\alpha$-down
and $\alpha$-up operators~by
\begin{equation} \label{alp_oper}
\begin{array}{l}
\!\! \dst \g a^{\downarrow}_\alpha(s)\!:=
e^{-\alpha \partial_s}\left( e^{\partial_s}
\sqrt{\sigma(s)} - \sqrt{\sigma(s)+\tau(s)}\,I
\right),
 \\[0.5cm]
\!\! \dst
\g a^{\uparrow}_\alpha(s)\!:=
\left(\sqrt{\sigma(s)} e^{-\partial_s} -
\sqrt{\sigma(s)+\tau(s)}\, I  \right) \!
e^{\alpha \partial_s},
\end{array}
\end{equation}
respectively.
\end{definition}

A straightforward calculation (by using the simple identity
$e^{\partial_s}\,\nabla = \Delta$) shows that for all
$\alpha\in\RR$
$$
\g h_1(s)=\g a^{\uparrow}_\alpha(s)\g a^{\downarrow}_\alpha(s),
$$
i.e., the operators $\g a^{\downarrow}_\alpha(s)$ and
$\g a^{\uparrow}_\alpha(s)$ factorize the Hamiltonian, defined
in (\ref{ham1}). Thus, we have the following
\begin{theorem}
Given a Hamiltonian  $\g h_1(s)$, defined by \refe{ham1}, the
operators $\g a^{\downarrow}_\alpha(s)$ and
$\g a^{\uparrow}_\alpha(s)$, defined in \refe{alp_oper}, are such
that for all $\alpha\in\CC$, the relation
$\g h_1(s)=\g a^{\uparrow}_\alpha(s)\g a^{\downarrow}_\alpha(s)$
holds.
\end{theorem}
\section{The dynamical algebra: The Charlier case}\label{sec-4}
Our next step is to find a dynamical symmetry algebra, associated
with the operator $\g h_1(s)$, or, equivalently, with the
corresponding family of polynomials, i.e.,
\textit{To find two operators $a(s)$ and $b(s)$, that factorize
the hamiltonian $\g h_1(s)$, i.e., $\g h_1(s)=b(s)a(s)$, and are
such that its commutator $[a(s),b(s)]=a(s)b(s)-b(s)a(s)=I$, where
$I$ denotes the identity operator.}

\begin{theorem} \label{el-teo}
Let  $\g h_1(s)$ be the hamiltonian, defined
in \refe{ham1}. The operators $b(s)=\g a^{\uparrow}_\alpha(s)$
and $a(s)=\g a^{\downarrow}_\alpha(s)$, given in  \refe{alp_oper},
factorize the Hamiltonian $\g h_1(s)$ \refe{ham1} and satisfy
the commutation relation $[a(s),b(s)]=\Lambda$ for a certain
complex number $\Lambda$, if and only if the following two
conditions hold:
\begin{equation} \label{first_cond_lin}
{\frac{\sigma(s-\alpha)[\sigma(s-\alpha)+\tau(s-\alpha)]}
{\sigma(s)[\sigma(s-1)+\tau(s-1)]}}=1
\end{equation}
and
\begin{equation} \label{second_cond_lin}
\sigma(s-\alpha+1)+ \sigma(s-\alpha)+\tau(s-\alpha)-
2\sigma(s)-\tau(s)=\Lambda.
\end{equation}
\end{theorem}

\begin{proof} Taking the expression for the operators
$a_\alpha^\uparrow(s)$ and $a_\alpha^\downarrow(s)$,
a straightforward
calculation shows that $a_\alpha^\downarrow(s)a_\alpha^\uparrow(s)=
A_1(s)e^{\partial_s}+A_2(s)e^{-\partial_s} +A_3(s)I$, where
\begin{equation}\label{help1}
\begin{array}{l}
\dst A_1(s)=
-\sqrt{\sigma(s+1-\alpha)[\sigma(s-\alpha+1)+\tau(s-\alpha+1)},
\\[0.5cm]
\dst A_2(s)=
-\sqrt{\sigma(s-\alpha)[\sigma(s-\alpha)+\tau(s-\alpha)},\\[0.5cm]
\dst A_3(s)=\sigma(s+1-\alpha)+\sigma(s-\alpha)+\tau(s-\alpha).
\end{array}
\end{equation}
In the same way,
$a_\alpha^\uparrow(s)a_\alpha^\downarrow(s)=\g h_1(s)=
B_1(s)e^{\partial_s}+B_2(s)e^{-\partial_s}+B_3(s)I$, where
\begin{equation}\label{help2}
\begin{array}{l}
\dst B_1(s)=-\nu(s),\quad
B_2(s)=-\nu(s-1),\quad
\dst B_3(s)=2\sigma(s)+\tau(s).
\end{array}
\end{equation}
Consequently,
\begin{equation}\label{help3}
[a_\alpha^\downarrow(s),a_\alpha^\uparrow(s)]=
\Big(A_1(s)-B_1(s)\Big)e^{\partial_s}+\Big(A_2(s)-
B_2(s)\Big)e^{-\partial_s} +\Big(A_3(s)-B_3(s)\Big)I.
\end{equation}
To eliminate the two terms in the right-hand side of \refe{help3},
which are proportional to $\exp(\pm \partial_s)$, one have to
require that $ A_1(s) - B_1(s)=0$ and $A_2(s)-B_2(s)=0$. But
${A_1(s)}/{B_1(s)}={A_2(s+1)}/{B_2(s+1)}$, hence, the requirement
that $A_1(s)=B_1(s)$ entails the relation $A_2(s)=B_2(s)$,
and vice versa. Thus, from \refe{help3} it follows that the
commutator $[a_\alpha^\uparrow(s),a_\alpha^\downarrow(s)]=\Lambda$,
iff 
$A_1(x)=B_1(s)$ and $A_3(s)-B_3(s)=\Lambda$.
\end{proof}

Using the main data for the discrete polynomials (see table
\ref{tabla-pol-dis}), we see that the only possible solution of
the problem 1 corresponds to the case when
$\sigma(s)+\tau(s)=const.$ and $\alpha=0$
i.e., the Charlier polynomials. Moreover, in this case
$\lambda_n=n$. 
\begin{corollary}
For the hamiltonian, associated with the Charlier polynomials,
$$
\g h_1^C(s)=-\sqrt{s\mu}\,e^{-\partial_s}-
\sqrt{(s+1)\mu}\,e^{\partial_s}+(s+\mu)I,
$$
$$
\g h_1^C(s)\Phi_n^C(s)=n \Phi_n^C(s),\qquad \Phi_n^C(s)=
\sqrt{\frac{e^{-\mu}\mu^{s-n}}{s!\,n!}}\,
C_n^\mu(s),\quad \mu>0,\quad n=0,1,2,\dots.
$$
Furthermore, the operators
\begin{equation}
\g a^{\downarrow}_0(s)=
 \sqrt{s+1}\,e^{\partial_s} - \sqrt{\mu} \,I, \qquad
 \dst \g a^{\uparrow}_0(s)=
\sqrt{s} e^{-\partial_s} -\sqrt{\mu}\,I,
\end{equation}
are such that $\g h_1^C=\b(s)\a(s)$ and
$[a_0^\downarrow(s),a_0^\uparrow(s)]=1$.
\end{corollary}

Notice that, since $\g h_1(s)\Phi(s) = \lambda \Phi(s)$,
\begin{equation*}\begin{split}
\g h_1(s) \{\a(s) \Phi(s)\}&=\b(s)\a(s)\{\a(s)\Phi(s)\}
=(\a(s)\b(s)-1)\{\a(s)\Phi(s)\}\\&= (\lambda-1)\{\a(s) \Phi(s)\},\\
\g h_1(s) \{\b(s) \Phi(s)\}  & = \b(s)\a(s)\b(s) \Phi(s)
=\b(s)(\lambda+1)\Phi(s)\\ & =(\lambda+1)\{\b(s) \Phi(s)\}.
\end{split}
\end{equation*}
In other words, if $\Phi(s)$ is an eigenvector of the hamiltonian
$\g h_1(s)$, then $\a(s) \Phi(s)$ is the eigenvector of $\g h_1(s)$,
associated with the eigenvalue $\lambda-1$, and $\b(s)\Phi(s)$ is
the eigenvector of $\g h_1(s)$, associated with the eigenvalue
$\lambda+1$. In general then $[\a]^k(s) \Phi(s)$ and
$[\b]^k(s)\Phi(s)$ are also eigenvectors corresponding to the
eigenvalues $\lambda-k$ and $\lambda+k$, respectively.\\

Using the preceding formulas for the Charlier polynomials, one finds
\begin{equation}
\b(s) \Phi_n^C(s)= U_n \Phi_{n+1}^C(s),\qquad
\a(s) \Phi_n^C(s)= D_n \Phi_{n-1}^C(s),
\label{cha-low-rai}
\end{equation}
where $U_n$ and $D_n$ are some constants.\\

If we now apply $\b(s)$ to the first equation of \refe{cha-low-rai}
and then use the second one and \refe{gen_ham2}, we find that
$\lambda_n=D_n U_{n-1}$. On the other hand, applying $\a(s)$
to the second equation in \refe{cha-low-rai} and using the first
one, as well as the fact that $\a(s)\b(s)\Phi_n^C(s)=
(\lambda_{n}+1)\Phi_n^C(s)$, one obtains that
$1+\lambda_n=U_nD_{n+1}=\lambda_{n+1}$, from which it follows
that $\lambda_n$ should be a linear function of $n$ (that
is also obvious from table \ref{tabla-pol-dis}).

If we use the boundary conditions $\sigma(s)\rho(s)\big|_{s=a,b}=0$,
as well as the formula of summation by parts, we obtain
$$
\pe{\g a^{\downarrow}_0(s)\Phi_{m}(s)}{\Phi_n(s)}_d=
\pe{\Phi_{m}(s)}{\g a^{\uparrow}_0(s)\Phi_{n}(s)}_d,
$$
i.e., the operators $\a(s)$ and $\b(s)$ are mutually adjoint.

{}From the above equality (the adjointness property) and
\refe{cha-low-rai} it follows that $D_{n+1}=U_n$, thus
$U_n^2=\lambda_{n+1}$, therefore
$U_n=\sqrt{\lambda_{n+1}}$ and $D_n=\sqrt{\lambda_{n}}$\,, i.e.,
we have the following

\begin{corollary}
The operators $\b(s)$ and $\a(s)$ are mutually adjoint
with respect to the inner product $\pe{\cdot}{\cdot}_d$ and
\begin{equation}\begin{split}
\b(s)\Phi_n^C(s)&=  \left(\sqrt{s}\,e^{\partial_s} -
\sqrt{\mu} \,I\right)
\Phi_n^C(s)= \sqrt{n+1}\, \Phi_{n+1}^C(s),\\[1mm]
\a(s)\Phi_n^C(s)&=  \left(\sqrt{s+1}\,e^{\partial_s} -
\sqrt{\mu} \,I\right)
\Phi_n^C(s)= \sqrt{n}\, \Phi_{n-1}^C(s).
\end{split}
\label{cha-low-rai1}
\end{equation}
\end{corollary}

{}From the above corollary one can deduce that
$$
\sqrt{s+1}\,\Phi_0^C(s+1)-\sqrt{\mu}\,\Phi_0^C(s)=0\quad\Ry
\quad \Phi_0^C(s)=N_0\sqrt{\frac{\mu^s}{s!}}.
$$
Using the orthonormality of $\Phi_0^C(s)$, one obtains that
$N_0=e^{-\mu/2}$. Thus
$$
\Phi_n^C(s)=\frac{1}{\sqrt{n!}}\, [\b(s)]^n \Phi_0^C(s)
=\frac{1}{\sqrt{n!}}
\left[\sqrt{s}\,e^{\partial_s} - \sqrt{\mu} \,I \right]^n
\left(\sqrt{\frac{e^{-\mu}\mu^s}{s!}}\right).
$$

Notice that
\begin{equation}\label{cha-rel-rec}
[\g h_1(s),\b(s)]=\sqrt \mu (\mu-1)+\mu \b(s), \qquad
[\g h_1(s),\a(s)]=-\sqrt \mu (\mu-1)-\mu \a(s).
\end{equation}

This example constitute a discrete analog of the quantum
harmonic oscillator \cite{ata98}.
\section{The dynamical algebra: The Meixner and Kravchuk cases}
\label{sec-5}
{}From the previous results we see that only the Charlier
polynomials (functions) have a closed simple oscillator algebra.
What to do in the other cases? To answer to this question, we can
use the following operators:
\begin{equation} \label{ope_a}
\begin{array}{l}
a(s)=\sqrt{\sigma(s+1)}\,e^{\half \partial_s}-
\sqrt{\sigma(s-1)+\tau(s-1)}\,e^{-\half \partial_s},\\[.4cm]
a^+(s)=\,e^{-\half \partial_s}\sqrt{\sigma(s+1)}-
\,e^{\half \partial_s}\sqrt{\sigma(s-1)+\tau(s-1)}.
\end{array}
\end{equation}
For this operators
$$
\g h_1(s) = a(s)a^+(s)+\tau'-\sigma'' .
$$
We will define a new hamiltonian $\g h_2(s)$ and operators
$b(s)$ and $b^+(s)$
\begin{equation} \label{ha2}
\g h_2(s)=C_a^2 \g h_1(s)+E, \qquad
b(s)=C_a a(s) \qquad \mbox{ and } \qquad b^+(s)=C_a a^+(s),
\end{equation}
where $C_a$ and $E$ are some constants (to be fixed later on).
Notice that from \refe{gen_ham2} it follows that the eigenfunctions
of $\g h_2(s)$ are the same functions \refe{nor-fun}, but the
eigenvalues are $C_a^2\lambda_n+E$, i.e.,
\begin{equation}\label{gen_ham3}
\g h_2(s)\Phi_n(s) =(C_a^2\lambda_n+E)\Phi_n(s).
\end{equation}
A straightforward computation yields
\begin{equation} \label{operac_1}
\g h_2(s)= b(s) b^+(s)+(\tau'-\sigma'')C_a^2+E,
\end{equation}
and
\begin{equation} \label{corch1}
\begin{split}
[b(s),b^+(s)]=& 
C_a^2 \sqrt{\sigma(s+\half)(\sigma(s-\tres)+\tau(s-\tres))} \,
e^{-\partial_s}\\ 
&+ \, C_a^2 \sqrt{\sigma(s+\tres)(\sigma(s-\half)+\tau(s-\half))}
\,e^{\partial_s} \\
&+ \g h_2(s)-C_a^2(2\sigma(s)+\tau(s)) +
\half(\tres \sigma''-\tau')C_a^2,
\end{split}
\end{equation}
or, equivalently,
\[
\begin{split}
[a(s),a^+(s)]=&
\sqrt{\sigma(s+\half)(\sigma(s-\tres)+\tau(s-\tres))}
\,e^{-\partial_s}\\ 
&+ \, \sqrt{\sigma(s+\tres)(\sigma(s-\half)+\tau(s-\half))}
\,e^{\partial_s} \\ 
&+ \g h_1(s)-(2\sigma(s)+\tau(s))I+\half(\tres \sigma''-\tau').
\end{split}
\]
The right-hand side of \refe{corch1} suggests us to use
the following new operators
\begin{equation} \label{ope_c}
\begin{array}{l}
c(s)=C_b b(s)\,e^{-\half \partial_s} \sqrt{\sigma(s+1)}=
C_b C_a(\sigma(s+1)-\,e^{-\partial_s}\nu(s)\,),\\[4mm]
c^+(s)=C_b \sqrt{\sigma(s+1)}\,e^{\half
\partial_s}b^+(s)=C_bC_a(\sigma(s+1)-\nu(s) \,e^{\partial_s}\ ),
\end{array}
\end{equation}
where, as before, $\nu(s)=\sqrt{\sigma(s+1)(\sigma(s)+\tau(s))}$.
So,
\begin{equation*}\begin{split}
&[\g h_2(s),c(s)]= -C_a^2(\sigma''\!\!-\!\tau')c(s)+
C_a C_b \left[\g h_2(s)+\left((\sigma''\!\!-\!\tau')C_a^2\!-\!
E\right)I\right]
\sigma'(s\!+\!\mbox{$\half$}),\\[0.2cm] \dst
&[\g h_2(s),c^+(s)]=C_a^2(\sigma''-\tau')c^+(s)-C_aC_b\sigma'
(s+\mbox{$\half$})[\g h_2(s)+\left((\sigma''-\tau')C_a^2-
E\right)I],\\[0.2cm]
&[c(s),c^+(s)]  \dst = C_a^2 C_b^2 \Big( \sigma'(s+\mbox{$\half$})
e^{-\partial_s}\nu(s)+\nu(s) e^{\partial _s}
\sigma'(s+\mbox{$\half$})-[\nu^2(s)-\nu^2(s-1)]I\Big).
\end{split}
\end{equation*}

The above expression leads to the following
\begin{theorem}
If  $\sigma''=0$, then the operators $\g h_2(s)$, $c(s)$ and
$c^+(s)$, defined by \refe{operac_1} and \refe{ope_c}, respectively,
form a closed algebra such that
\begin{equation*}
\begin{split}
[\g h_2(s),c(s)]&
=\tau' C_a^2 c(s)+
C_bC_a\sigma'(0)\left(\g h_2(s)-\tau'C_a^2-E\right),\\[0.2cm]
[\g h_2(s),c^+(s)]&=-\tau' C_a^2 c^+(s)-C_bC_a \sigma'(0)
\left(\g h_2(s)-\tau'C_a^2-E\right),\\[0.2cm]
[c(s),c^+(s)]&=C_b^2\Big[\sigma(s)-\sigma'(0)(\g h_2(s)-E)\Big].
\end{split}
\end{equation*}
\end{theorem}
Observe also that with this particular choice
$\sigma'(s\!+\!\mbox{$\half$})=\sigma'(0)$ and
$$
\begin{array}{c}
c(s)+c^+(s)=C_bC_a(\g h_1(s)+2\sigma'(0)+\tau(s)),\\[0.3cm]
\nu^2(s)-\nu^2(s-1)=\sigma'(0)(2\sigma(s)+\tau(s))+\tau'\sigma(s).
\end{array}
$$
Furthermore, using the boundary conditions
$\sigma(s)\rho(s)\big|_{s=a,b}=0$,
one finds
\begin{equation*}\begin{split}
\pe{c\,\Phi_n}{\Phi_m}_d&=
C_aC_b\sum_{s=a}^{b-1}\sigma(s+1)\Phi_n(s)\Phi_m(s)-\sum_{s=a}^{b-1}
\nu(s-1)\Phi_{n}(s-1)\Phi_m(s)\\ &=
C_aC_b\sum_{s=a}^{b-1}\sigma(s+1)\Phi_n(s)\Phi_m(s)-C_aC_b
\sum_{s=a}^{b-1}
\nu(s)\Phi_{n}(s)\Phi_m(s+1)\\ &=\pe{\Phi_n}{c^+\Phi_m}_d,
\end{split}
\end{equation*}
i.e., the following theorem follows.

\begin{theorem}
The operators $c(s)$ and $c^+(s)$ are mutually adjoint.
\end{theorem}

Notice also that the operators $\g h_1(s)$ and $\g h_2(s)$ are
selfadjoint operators.

\begin{nota}
Since $\lambda=\lambda_n=-n(\tau'+(n-1)\sigma''/2)$, the identity
$\sigma''=0$ is equivalent to the statement that $\lambda_n$
is a linear function of $n$. In this case $\lambda_n=-n\tau'$.
\end{nota}

In the following we will consider only the case when $\sigma''=0$,
i.e., the case of the Meixner, the Kravchuk and the Charlier
polynomials.

If we define the operators
\begin{equation} \label{ope_k}
\begin{split}
K_0(s)&=\,\g h_2(s)(-\tau'C_a^2)^{-1}\\[.2cm]
K_-(s)&=-\tau'C_a^2\,c(s)-C_bC_a \sigma'(0)\left(\g h_2(s)
-\tau'C_a^2-E\right), \\[0.2cm]
K_+(s)&=-\tau'C_a^2\,c^+(s)-C_bC_a \sigma'(0)
\left(\g h_2(s)-\tau'C_a^2-E\right),
\end{split}
\end{equation}
then
$$
[K_0(s),K_\pm(s)]=\pm K_\pm(s) \
\mbox{ y } \ [K_-(s),K_+(s)]=A_0K_0(s)+A_1,
$$
where
$$
\begin{array}{rl}
\dst A_0&=-2\tau'\sigma'(0)C_b^2C_a^4(-\tau'C_a^2)
(\sigma'(0)+\tau')\quad
\mbox{ and } \\[0.4cm]
\dst A_1&=-E A_0(-\tau'C_a^2)^{-1}+
C_b^2C_a^6\tau'^2[\sigma'(0)\tau(0)-\sigma(0)\tau'].
\end{array}
$$

The case $A_0=0$ corresponds to the Charlier case (see the
previous section). If $A_0\neq0$, we have two possibilities:
$A_0>0$ and $A_0<0$. In the following we will choose
$C_a^2=-1/\tau'$, i.e., $-\tau'C_a^2=1$.

In the first case $A_0>0$ one can choose $C_b$ and $E$ in such a way
that $A_0=2$ and $A_1=0$. Thus
\begin{equation}\label{Cb-E1}
C_b^2=\frac{-\tau'}{\sigma'(0)[\tau'+\sigma'(0)]},
\qquad
E=-\frac{C_b^2[\sigma'(0)\tau(0)-\sigma(0)\tau']}{2\tau'}.
\end{equation}
Consequently, the operators $K_{\pm}$ and $K_0$ are such that
\begin{equation}\label{com-rel}
[K_0(s),K_\pm(s)]=\pm K_\pm(s)
\ \mbox{ and }
\ [K_-(s),K_+(s)]=2K_0(s).
\end{equation}
This case corresponds to the Lie algebra Sp$(2,\g R)$.

In the second case one can choose $C_b$ and $E$ in such a way that
$A_0=-2$ and $A_1=0$. Thus
\begin{equation}\label{Cb-E2}
C_b^2=\frac{\tau'}{\sigma'(0)[\tau'+\sigma'(0)]},\qquad
E=\frac{C_b^2[\sigma'(0)\tau(0)-\sigma(0)\tau']}{2\tau'}.
\end{equation}
Consequently, the operators $K_{\pm}$ and $K_0$ are such that
\begin{equation}\label{com-rel2}
[K_0(s),K_\pm(s)]=\pm K_\pm(s)
\ \mbox{ and }
\ [K_+(s),K_-(s)]=2K_0(s).
\end{equation}
This case corresponds to the Lie algebra so$(3)$.\\

Notice that since the operator $\g h_2(s)$ is selfadjoint, the
operators $K_\pm(s)$ are mutually adjoint in both cases, i.e.
$$
\pe{K_+\Phi_m}{\Phi_n}_d=
\pe{\Phi_m}{K_-\Phi_n}_d.
$$
\subsection{Dynamical symmetry algebra Sp$(2,\g R)$}
Let us consider the first case. We start with the operator
\begin{equation} \label{casop1}
K^2(s)=K^2_0(s)-K_0(s)-K_+(s)K_-(s),
\end{equation}
where $K_0(s)$, $K_+(s)$, and $K_-(s)$ are the operators
given in \refe{ope_k}.
A straightforward calculation gives
$$
K^2(s) = E(E-1)I,\qquad
E=\frac{\tau(0)\sigma'(0)-\tau'\sigma(0)}{2\sigma'(0)
(\sigma'(0)+\tau')},
$$
where $E$ is given by  \refe{Cb-E1},
i.e., the $K^2(s)$ is the invariant Casimir operator.

Furthermore, if we define the normalized functions
$$
\Phi_n(s)=\sqrt{\frac{\rho(s)}{d^2_n}}P_n(s),
$$
we have
\begin{equation} \label{casop2}
\begin{array}{ll}
K^2(s)\Phi_n(s) = E(E-1)\Phi_n(s), & \qquad
K_0(s)\Phi_n(s)=(n+E)\Phi_n(s).
\end{array}
\end{equation}
Now using the commutation relation \refe{com-rel}, it is easy to
show that
$$
K_0(s)[K_\pm(s)\Phi_n(s)]=(n+E\pm1) K_\pm(s)\Phi_n(s).
$$
Consequently, from \refe{casop2} and the above equation
we deduce that
$$
\begin{array}{l}
K_+(s) \Phi_n(s)=\widetilde\kappa_{n} \Phi_{n+1}(s),\qquad
K_-(s) \Phi_n(s)=\kappa_n \Phi_{n-1}(s).
\end{array}
$$
Employing the mutual adjointness of the operators $K_{\pm}$, one
obtains
$$
\widetilde\kappa_{n}=\pe{K_+\Phi_n(s)}{\Phi_{n+1}(s)}_d=
\pe{\Phi_n(s)}{K_-\Phi_{n+1}(s)}_d=\kappa_{n+1},
$$
thus
\begin{equation} \label{lrex}
\begin{array}{l}
K_+(s) \Phi_n(s)=\kappa_{n+1} \Phi_{n+1}(s),\qquad
K_-(s) \Phi_n(s)=\kappa_n \Phi_{n-1}(s).
\end{array}
\end{equation}
In order to compute $\kappa_n$, use \refe{casop2} and \refe{lrex};
this yields
$$
E(E-1)=(n+E)^2-(n+E)-\kappa_n^2\quad\Ry\quad
\kappa_n=\sqrt{n(n+2E-1)}.
$$

In this case the functions $(\Phi_n)_n$ define a basis for the
irreducible unitary representation $D^+(-E)$ of the Lie group
(algebra)  Sp$(2,\g R)$.

{}From the above formula it follows that the functions $\Phi_n(s)$
can be obtained recursively via the application of the operator
$K_+(s)$, i.e.,
$$
\Phi_n(s)=\frac{1}{\kappa_1 \cdots \kappa_n}
K^n_+(s)\Phi_0(s), \quad \Phi_0(s)=
\frac {\sqrt{\rho(s)}}{d_0}.
$$
where $\rho(s)$ is the weight function of the corresponding
orthogonal polynomial family and $d_0$ is the norm of the $P_0(s)$.

\subsubsection{Example: The Meixner functions}
\par
Let consider the Meixner functions
$$
\Phi_n^M(s)=\mu^{(s-n)/2}(1-\mu)^{\gamma/2+n}
\sqrt{\frac{(\gamma)_s}{s!n!(\gamma)_n}}\,
M_n^{\gamma,\mu}(s),\quad
n\geq0,
$$
and the hamiltonian $\g h_1(s)$
$$
\g h_1^M(s)= -\sqrt{\mu s(s + \gamma-1)} e^{-\partial_s} -
\sqrt{\mu(s+1)(s + \gamma)} e^{\partial_s} +
\left(s + \mu\left( s + \gamma  \right)\right) I,
$$
thus $\g h_1^M(s)\Phi_n^M(s)=n \Phi_n^M(s)$. In this case we have
$C_b=\sqrt{\frac{1-\mu}{\mu}}$, $E=\frac\gamma2$,
$C_a=\sqrt{\frac1{1-\mu}}$.
Therefore
$$
b(s)=-\sqrt{\frac{(s-1 + \gamma)\mu}{1-\mu}}
e^{-\frac12\partial_s}+\sqrt{\frac{s+1}{1 - \mu}}
e^{\frac12\partial_s},
$$
$$
b^+(s)=-\sqrt{\frac{(s-\frac12 + \gamma)\mu}{1-\mu}}
e^{\frac12\partial_s}+\sqrt{\frac{s+\frac12}{1 - \mu}}
e^{-\frac12\partial_s}.
$$
Consequently,
$$
\g h_2(s)=\frac1{1-\mu}\,\g h_1(s)+\frac\gamma2=
b(s)b^+(s)+\frac\gamma2-1.
$$
Moreover,
$$
\g h_2(s)\Phi_n^M(s)=\left(n+\frac\gamma2\right)\Phi_n^M(s),
$$
$$
K_0(s)=
-\sqrt{s( s-1 + \gamma) }\frac{\sqrt{\mu }}{1-\mu}\,e^{-\partial_s}
-\sqrt{(s+1)(s + \gamma)}\frac{\sqrt{\mu }}{1-\mu}\,e^{\partial_s}
+\left( s + \frac{\gamma }{2}\right)\frac{1+\mu}{1-\mu}   \,I,
$$
$$
K_+(s)=
-\frac{\sqrt{s( s-1 + \gamma)}}{1-\mu}\,e^{-\partial_s}
-\frac{\mu\sqrt{(s+1)(s + \gamma)}}{1-\mu}\,e^{\partial_s}
+\frac{ \sqrt{\mu}}{1-\mu}
\left( 2s + \gamma  \right)\,I,
$$
$$
K_-(s)=
-\frac{\mu\sqrt{s( s-1 + \gamma)}}{1-\mu}\,e^{-\partial_s}
-\frac{\sqrt{(s+1)(s + \gamma)}}{1-\mu}\,e^{\partial_s}
+\frac{ \sqrt{\mu}}{1-\mu}
\left( 2s + \gamma  \right)\,I,
$$
and
$$
K_0(s)\Phi_n^M(s)=\left(n+\frac\gamma2\right)
\Phi_n^M(s),\qquad
K^2(s)\Phi_n^M(s)=\frac\gamma2\left(\frac\gamma2-1\right)
\Phi_n^M(s),
$$
\begin{equation}\label{k+-mei}
\begin{array}{l}
K_+(s)\Phi_n^M(s)=\sqrt{(n+1)(n+\gamma)}
\Phi_{n+1}^M(s),\\[5mm]
K_-(s)\Phi_n^M(s)=\sqrt{n(n+\gamma-1)}
\Phi_{n-1}^M(s).
\end{array}
\end{equation}
Using the fact that $\g h_2\Phi_0^M(s)=\frac\gamma2\Phi_0^M(s)$,
together with the formulas \refe{ope_k} and \refe{k+-mei}, one finds
$$
0=K_-(s)\Phi_0^M(s)=\sqrt{s(s-1+\gamma)}\,\Phi_0^M(s-1)-
s\mu^{-\frac12}\Phi_0^M(s),
$$
therefore the normalized function $\Phi_0^M(s)$ is
$$
\Phi_0^M(s)=\sqrt\frac{(1-\mu)^{\gamma}}{\Gamma(\gamma)}
\sqrt{\frac{\mu^s\Gamma(\gamma+s)}
{\Gamma(\gamma)\Gamma(s+1)}},
$$
and
$$
\Phi_n^M(s)={\sqrt\frac{(1-\mu)^{\gamma}}{n!\Gamma(\gamma+n)}}
\left(K_+(s)\right)^n
\left[\sqrt{\frac{\mu^s\Gamma(\gamma+s)}
{\Gamma(s+1)}}\right].
$$

A similar result have been obtained before in \cite{ata98}.

\subsection{Dynamical symmetry algebra so$(3)$}
Let us consider the second case and define the following operator
\begin{equation} \label{casop3}
K^2(s)=K^2_0(s)+K_0(s)+K_-(s)K_+(s).
\end{equation}
where $K_0(s)$, $K_+(s)$, and $K_-(s)$ are the operators given in
\refe{ope_k}. Substituting the value of $E$, given by
\refe{Cb-E2}, and doing some straightforward computations
yield
$$
K^2(s) = E(E-1)I,\qquad
E=\frac{\tau(0)\sigma'(0)-\tau'\sigma(0)}{2\sigma'(0)
(\sigma'(0)+\tau')},
$$
i.e., the $K^2(s)$ is the invariant Casimir operator.

Moreover, if we define the normalized functions as
$$
\Phi_n(s)=\sqrt{\frac{\rho(s)}{d^2_n}}P_n(s),
$$
we have
\begin{equation} \label{casop4}
\begin{array}{ll}
K^2(s)\Phi_n(s) = E(E-1)\Phi_n(s), &\qquad
K_0(s)\Phi_n(s)=(n+E)\Phi_n(s).
\end{array}
\end{equation}
Now using the commutation relation \refe{com-rel2}, we have
$$
K_0(s)[K_\pm(s)\Phi_n(s)]=(n+E\pm1)
K_\pm(s)\Phi_n(s).
$$
Consequently, from \refe{casop2} and the above equation, we conclude
that
$$
\begin{array}{l}
K_+(s) \Phi_n(s)=\widetilde\kappa_{n} \Phi_{n+1}(s),\qquad
K_-(s) \Phi_n(s)=\kappa_n \Phi_{n-1}(s).
\end{array}
$$
Using the mutual adjointness of the operators $K_{\pm}$, one obtains
$$
\widetilde\kappa_{n}=\pe{K_+\Phi_n(s)}{\Phi_{n+1}(s)}_d=
\pe{\Phi_n(s)}{K_-\Phi_{n+1}(s)}_d=\kappa_{n+1},
$$
thus
\begin{equation} \label{lrex2}
\begin{array}{l}
K_+(s) \Phi_n(s)=\kappa_{n+1} \Phi_{n+1}(s),\qquad
K_-(s) \Phi_n(s)=\kappa_n \Phi_{n-1}(s).
\end{array}
\end{equation}
To compute $\kappa_n$, use \refe{casop2} and \refe{lrex2}; this
leads to
$$
E(E-1)=(n+E)^2+(n+E)+\kappa_{n+1}^2 \quad\Ry\quad
\kappa_{n}=\sqrt{-(n)(n+2E-1)}.
$$

In this case the functions $(\Phi_n)_n$ define a basis for the
irreducible unitary representation $D^+(-E)$ of the Lie
algebra so$(3)$.

As in the previous case, from the above formula it follows
that the functions $\Phi_n(s)$ can be obtained recursively
via the application of the operator $K_+(s)$, i.e.,
$$
\Phi_n(s)=\frac{1}{\kappa_1 \cdots \kappa_n}
K^n_+(s)\Phi_0(s), \quad \Phi_0(s)=
\frac {\sqrt{\rho(s)}}{d_0}.
$$
where $\rho(s)$ is the weight function for the associated
orthogonal polynomial family and $d_0$ is the norm of the $P_0(s)$.

\subsubsection{Example: The Kravchuk functions}

\par
Let us consider now the Krav\-chuk functions
$$
\Phi_n^K(s)=p^{(s-n)/2}(1-p)^{(N-n-s)/2}\sqrt{\frac{n!(N-n)!}{s!(N-s)!}}
K_n^p(s,N),\quad
0\leq n\leq N,
$$
and the corresponding hamiltonian $\g h_1(s)$
$$
\g h_1^K(s)=-\frac{{\sqrt{ps(N- s+1)}}}{{\sqrt{1 - p}}}
e^{-\partial_s}  +  \frac{Np+ s-2ps}{1-p}I -
\frac{{\sqrt{p(s+1)(N-s)}}}{{\sqrt{1-p}}}e^{\partial_s}\,,
$$
thus $\g h_1^K(s)\Phi_n^K(s)=n \Phi_n^K(s)$. In this case
$C_a=\sqrt{1-p}$, $C_b=\sqrt{p^{-1}}$, $E=-\frac N2$,  therefore
$$
b(s)=-\sqrt{p(N-s+1)}e^{-\half\partial_s}+\sqrt{(1-p)(s+1)}
e^{\half\partial_s},
$$
$$\begin{array}{l}
b^+(s)=-\sqrt{p(N-s+\frac12)}e^{\half\partial_s}+
\sqrt{(1-p)(s+\frac12)}
e^{\half\partial_s}.
\end{array}
$$
Consequently,
$$
\g h_2(s)= (1-p) \,\g h_1(s)-\frac N2= b(s)b^+(s)-\frac N2-1.
$$
Moreover,
$$
\g h_2(s)\Phi_n^K(s)=\left(n-\frac N2\right)\Phi_n^K(s),
$$
$$\begin{array}{l}
K_0(s)=-\sqrt{p(1-p)s(N- s+1)}
e^{-\partial_s}  -\sqrt{p(1-p)(s+1)(N-s)}e^{\partial_s}\\[2mm]
\qquad\qquad +[N(p-\half)- s(2p-1)]I,\\[4mm]
K_+(s)=(1\!-\!p){{\sqrt{s(N- s+1)}}}e^{-\partial_s}+p
{{\sqrt{(s+1)(N-s)}}}e^{\partial_s}-\sqrt{p(1\!-\!p)}(2s-N)I,\\[4mm]
K_-(s)=p{{\sqrt{s(N- s+1)}}}e^{-\partial_s}+(1\!-\!p)
{{\sqrt{(s+1)(N-s)}}}e^{\partial_s}-\sqrt{p(1\!-\!p)}(2s-N)I,
\end{array}
$$
and
$$
K_0(s)\Phi_n^K(s)=\left(n-\frac N2\right)
\Phi_n^K(s),\qquad
K^2(s)\Phi_n^K(s)=\frac N4\left(N+2\right)
\Phi_n^K(s),
$$
\begin{equation}\label{k+-kra}
\begin{array}{l}
K_+(s)\Phi_n^K(s)=\sqrt{(n+1)(N-n)}
\Phi_{n+1}^K(s),\\[3mm]
K_-(s)\Phi_n^K(s)=\sqrt{n(N-n+1))}
\Phi_{n-1}^K(s).
\end{array}
\end{equation}
Using the fact that $\g h_2\Phi_0^K(s)=-\frac N2\Phi_0^K(s)$,
together with the formulas \refe{ope_k} and \refe{k+-kra}, we find
$$
0=K_-(s)\Phi_0^K(s)=\sqrt{\frac{p}{1-p}} \left(
s\Phi_0^K(s)-\sqrt{\frac{ps(N-s+1)}{1-p}}\Phi_0^K(s-1)\right),
$$
therefore the normalized function $\Phi_0^K(s)$ is equal to
$$
\Phi_0^K(s)=p^{(s-n)/2}(1-p)^{(N-n-s)/2}\sqrt{\frac{n!(N-n)!}{s!(N-s)!}},
$$
and
$$
\Phi_n^K(s)= \sqrt{\frac{(N-n)!(1-p)^{N-n}}{N!p^n}}
\left(K_+(s)\right)^n
\left[{N\choose s}\left(\frac{p}{1-p}\right)^s\right].
$$
\section{The $q$-case}
\label{sec-6} To conclude this paper we will discuss here briefly
what happens in the $q$-case. The preliminary results, related
with this case, have been presented during the Bexbach Conference
2002 \cite{ran03}. A more detailed exposition of these results is
under preparation.

One can first introduce the corresponding normalized functions
\begin{equation}
\Phi_n(s)= \frac{A(s)\sqrt{\rho(s)}}{d_n} P_n(s;q),
\label{nor-fun-q}
\end{equation}
where $d_n$ is the norm of the $q$-polynomials $P_n(s;q)$, $\rho(s)$
is the solution of the Pearson-type equation
$$
\frac{\Delta}{\Delta x(s-\frac{1}{2})}
\left[\sigma(s) \rho(s)\right]= \tau(s) \rho(s)\quad\mbox{or}
\quad \sigma(s+1)\rho(s+1)=\sigma(-s-\mu)\rho(s),
$$
and $A(s)$ is an arbitrary continuous function, not vanishing in the
interval $(a,b)$ of orthogonality of $P_n$. If $P_n(s;q)$ possess
the discrete orthogonality property  \refe{dis-ort}, then the
functions $\Phi_n(s)$ satisfy
\begin{equation}
\pe{\Phi_n(s)}{\Phi_m(s)}=\sum_{s=a}^{b-1}{\Phi_n(s)}{\Phi_m(s)}
\frac{\nabla x_1(s)}{A^2(s)}=\delta_{n,m}.
\label{ort-Phi-q}
\end{equation}
Notice that if $A(s)=\sqrt{\nabla x_1(s)}$, then the set
$(\Phi_n)_n$ is an orthonormal set. Obviously, in the case of a
continuous orthogonality (as for the Askey-Wilson polynomials) one
needs to change the sum in \refe{ort-Phi-q} by a Riemann integral
\cite{arsus,nsu}.

Next, we define the $q$-Hamiltonian $\g H_q(s)$ of the form
\begin{equation} \label{gen_hamilt}
\g H_q(s):=
\frac{1}{\nabla x_1(s)}A(s)H_q(s)\frac{1}{A(s)},
\end{equation}
where
\begin{equation}\begin{split}\label{ham3}
\dst H_q(s) := & -\frac{\sqrt{\sigma(-s\!-\!\mu\!+\!1)\sigma(s)}}
{\nabla x(s)}e^{-\partial_s}
-\frac{\sqrt{\sigma(-s\!-\!\mu)\sigma(s+1)}}{\Delta x(s)}
e^{\partial_s} \\ & +\left(\frac{\sigma(-s\!-\!\mu)}{\Delta
x(s)} +\frac{\sigma(s)}{\nabla x(s)} \right) I.
\end{split}\end{equation}
As in the previous case, one can easily check that
\begin{equation} \label{gen_ham2-q}
\g H_q(s)\Phi_n(s) =\lambda_n \Phi_n(s).
\end{equation}

Now we define the $\alpha$ operators:
\begin{definition}
Let $\alpha$ be a real number and $A(s)$ and $B(s)$ are two
arbitrary continuous non-vanishing functions. We define a family
of $\alpha$-down and $\alpha$-up operators by
\begin{equation} \label{alp_oper-q}
\begin{array}{l}
\!\!\dst\g a^{\downarrow}_\alpha(s)\!:=\!
\frac{B(s)}{\sqrt{\nabla x_1(s)}}e^{-\alpha \partial_s}
\!\left( e^{\partial_s}\sqrt{\frac{\sigma(s)}{\nabla x(s)}} -
\sqrt{\frac{\sigma(-s-\mu)}{\Delta x(s)}}
\right)\!\frac{1}{A(s)},
 \\[0.6cm]
\!\! \dst \g a^{\uparrow}_\alpha(s)\!:=
\!\frac{1}{\nabla x_1(s)} A(s)
\!\left( \sqrt{\frac{\sigma(s)}{\nabla x(s)}} e^{-\partial_s} -
\sqrt{\frac{\sigma(-s-\mu)}{\Delta x(s)}}  \right) \!
e^{\alpha \partial_s} \frac{\sqrt{\nabla x_1(s)}}{B(s)},
\end{array}
\end{equation}
respectively.
\end{definition}
The first result in this case is \cite{ran03}:
\begin{theorem}
Given a $q$-Hamiltonian \refe{gen_hamilt} $\g H_q(s)$, then the
operators $\g a^{\downarrow}_\alpha(s)$ and
$\g a^{\uparrow}_\alpha(s)$,
defined in  \refe{alp_oper}, are such that for all $\alpha\in\CC$,
$\g H_q(s)=\g a^{\uparrow}_\alpha(s)\g a^{\downarrow}_\alpha(s)$.
\end{theorem}

Our next step is again to find a dynamical symmetry algebra, associated
with the operator $\g H_q(s)$, or equivalently, with the corresponding
family of $q$-polynomials.

\begin{definition}\label{def-com}
Let $\varsigma$ be a complex number, and let $a(s)$ and $b(s)$ be
two operators. We define the $\mathbf{\varsigma}${\em -commutator
of $a$ and $b$} as
\begin{equation}\label{comm}
[a(s),b(s)]_{\varsigma} = a(s)b(s)-\varsigma b(s) a(s).
\end{equation}
\end{definition}

We want to know whether the following problem:
{\em To find two operators $a(s)$ and $b(s)$ and a constant $\q$
such that the Hamiltonian $\g H_q(s)=b(s)a(s)$ and
$[a(s),b(s)]_{\q}=I$,} has a non-trivial solution.\\

Obviously, we already know the answer to the first part: these are
the operators $b(s)=\g a^{\uparrow}_\alpha(s)$ and
$a(s)=\g a^{\downarrow}_\alpha(s)$, given in  \refe{alp_oper-q}.
The answer to the second part of this problem is summarized in the
following two theorems (in what follows we assume that $A(s)=B(s)$).

\begin{theorem}{\cite{ran03}}\label{el-teo-q-lam}
Let $(\Phi_n)_n$ be the
eigenfunctions of  $\g H_q(s)$, corresponding to the eigenvalues
$(\lambda_n)_n$, and suppose that the problem 1 has a solution for
$\Lambda\neq0$. Then the eigenvalues $\lambda_n$ of the difference
equation \refe{gen_ham2} are $q$-linear or $q^{-1}$-linear functions
of $n$, i.e., $\lambda_n=C_1 q^n+C_3$ or $\lambda_n=C_2q^{-n}+C_3$,
respectively.
\end{theorem}
\begin{theorem}{\cite{ran03}}
 \label{el-teo-q-lin} Let  $\g H_q(s)$ be the
$q$-Hamiltonian \refe{gen_hamilt}.
The operators $b(s)=\g a^{\uparrow}_\alpha(s)$ and $a(s)=\g
a^{\downarrow}_\alpha(s)$, given in  \refe{alp_oper-q} with
$B(s)=A(s)$,  factorize the Hamiltonian $\g H_q(s)$
\refe{gen_hamilt}
and satisfy the commutation relation $[a(s),b(s)]_{\q}=\Lambda$ for
a certain complex number $\q$, if and only if the following two
conditions hold
\begin{equation} \label{first_cond_lin-q}
\frac{\nabla x(s)}{\nabla x_1(s-\alpha)} \sqrt{\frac{\nabla
x_1(s-1)\nabla x_1(s)}{\nabla x(s-\alpha)\Delta x(s-\alpha)}}
\sqrt{\frac{\sigma(s-\alpha)\sigma(-s-\mu+\alpha)}{\sigma(s)
\sigma(-s-\mu+1)}}=\q,
\end{equation}
\begin{equation} \label{second_cond_lin-q}
\begin{small}\frac{1}{\Delta
x(s\!-\!\alpha)}\!\left(\frac{\sigma(s\!-\!\alpha+1)}{\nabla
x_1(s\!-\!\alpha+1)}+ \frac{\sigma(\!-\!s\!-\!\mu+\alpha)}{\nabla
x_1(s\!-\!\alpha)}\right)\!-\varsigma \frac{1}{\nabla
x_1(s)}\left(\!\frac{\sigma(s)}{\nabla
x(s)}+\frac{\sigma(\!-\!s\!-\!\mu)}{\Delta x(s)}\!\right)=\Lambda.
\end{small}
\end{equation}
\end{theorem}
The proof of the theorem \ref{el-teo-q-lin} is similar to the
proof of the theorem \ref{el-teo}, presented here for the case
of the uniform lattice $x(s)=s$.

Let us point out that the $q$(respectively, $q^{-1}$)-linearity
of the eigenvalues is a necessary condition in order to provide
that the solution of problem 1 exists. But this condition is not
sufficient. For example, if we take the discrete $q$-Laguerre
polynomials $L_n^\alpha (x;q)$ (for more details see \cite{ran03})
with $a\neq q^{-1/2}$, the problem  has not a solution, but
$\lambda_n$ is a $q$-linear function of $n$.

\subsection{Examples}
We present here only two examples, others can be found in
\cite{ran03}.
\subsubsection{The Al-Salam \& Carlitz I $q$-polynomials
$U_n^{(a)}(x;q)$}
We start with the very well known case: the Al-Salam \& Carlitz I
polynomials \cite{ks}. The corresponding normalized functions
\refe{nor-fun} are
$$
\Phi_n(x)=\sqrt{\frac{(qx,a^{-1}qx;q)_\infty(-a)^nq^{{n\choose 2}}}
{(1-q)(q;q)_n(q,a,q/a;q)_\infty}A^2(s)}
\, _2\phi_1\left(\begin{array}{c|c} q^{-n},x^{-1} \\[-0.3cm] &
q;\dst \frac{qx}{a} \\[-0.2cm]0 \end{array} \right),
\quad x:=q^s.
$$
Putting $A(s)=\sqrt{\nabla x_1(s)}=\sqrt{x k_q}$, where
$k_q=q^{\half}-q^{-\half}$, we have that the
functions $(\Phi_n)_n$ satisfy the orthogonality condition
$$
\int_a^1 \Phi_n(x)\Phi_m(x)\frac{d_q x}{k_q x}=\delta_{n,m},
$$
where the integral $\int_a^1 f(x)d_qx$ denotes the classical
Jackson $q$-integral.

For these polynomials $\sigma(s)+\tau(s)\nabla x_1(s)=a$, therefore
$\alpha=0$ and  $\varsigma=q^{-1}$. The $q$-Hamiltonian has the form
$$\begin{array}{rl}
\g H_q(s)= & \dst
-{\frac{{q^2}\sqrt{a(x-1)(x-a)}}
{{{\left(q -1 \right) }^2}{x^2}}}e^{-\partial_s}
-\frac{\sqrt{a(1-qx)(a-qx)}}{x^2}e^{\partial_s} \\[5mm]
& \dst +\left({\frac{{\sqrt{q}}\left( q\left(x -1\right) x
+ a\left( 1 + q - qx \right)  \right)}
{{{\left(q -1 \right) }^2}{x^2}}}\right)I,\quad x=q^s.
\end{array}
$$
Consequently, $\g H_q(s)\Phi_n(s)=q^{\frac{3}{2}}
\frac{1-q^{-n}}{(1-q)^2}\Phi_n(s)$ and the operators ($x=q^s$)
\begin{equation} \label{op_ex_3}
\begin{array}{l}
\dst \g a^{\downarrow}(s) \equiv \g
a^{\downarrow}_0(s)=\frac{q^{\racion{1}{4}}}{k_q x}\left(
\sqrt{(x-1/q)(x-a/q)} \ e^{\partial_s}-\sqrt{a}\ I\right),\\[0.6cm]
\dst \g a^{\uparrow}(s) \equiv \g a^{\uparrow}_0(s)=
\frac{q^{\racion{1}{4}}}{k_q x}\left(
\sqrt{(x-1)(x-a)} \ e^{-\partial_s}-\sqrt{a/q}\ I\right),
\end{array}
\end{equation}
are such that
$$
\g a^{\uparrow}(s)\g a^{\downarrow}(s)=\g H_q(s),\quad\mbox{and}
\quad[\g a^{\downarrow}(s),\g a^{\uparrow}(s)]_{q^{-1}}=
\frac{1}{k_q}.
$$
A straightforward calculation shows that the operators
$\g a^{\uparrow}(s)$ and $\g a^{\downarrow}(s)$ are mutually
adjoint. A similar factorization was obtained earlier in
\cite{asuslov2} and more recently in \cite{ran-rob} with
the aid of a different technique.

The special case of the  Al-Salam \& Carlitz I polynomials are the
discrete $q$-Hermite I $h_n(x;q)$, $x=q^s$, polynomials, which
correspond to the parameter $a=-1$ \cite{ks}.

\subsection{Continuous $q$-Hermite polynomials}
Let us now consider the particular case of the Askey-Wilson
polynomials when all their parameters are equal to zero, i.e.,
the continuous $q$-Hermite polynomials \cite{ks}. In this case
$\sigma(s)=C_\sigma q^{2s}$.

Let us choose $A(s)=B(s)=\sqrt{\nabla x_1(s)}$. In this case
$\alpha=1/2$ and $\varsigma=1/q$. The corresponding Hamiltonian
is given by
\begin{small}
$$
\begin{array}{c}
\dst \!\!\!\g H_q(s)\!=
\!\frac{-1}{k_q^2\sqrt{\sin \theta}}\! \left(
\!\frac{C_\sigma q}{\sin (\theta\! +\!\racion{i}{2}\log q)
\sqrt{\sin (\theta\!+\!i\log q)}}\, e^{\!-\partial_s}\!\!\!+\!
\frac{C_\sigma q}{\sin (\theta \!-\!\racion{i}{2}\log q)
\sqrt{\sin (\theta\!-\!i\log q)}} \,e^{\partial_s}\!\right) \\[0.6cm]
+\dst \frac{1}{k_q^2 \sin \theta}\left(\frac{C_\sigma q^{2s}}
{\sin (\theta +\racion i 2 \log q)}+
\frac{C_\sigma q^{-2s}}{\sin (\theta -\racion i 2 \log q)}\right)I,
\end{array}
$$
\end{small}
and the $\alpha$-operators
$$
\begin{array}{l}
\dst a_{1/2}^{\downarrow}(s)=
e^{\half \partial_s}\sqrt{\frac{C_\sigma q^{2s}}{-k^2_q \sin \theta
\sin (\theta +\racion i 2 \log q)}}
- e^{-\half \partial_s}\sqrt{\frac{C_\sigma q^{-2s}}{-k^2_q
\sin \theta \sin (\theta -\racion i 2 \log q)}},\\[0.6cm]
\dst a_{1/2}^{\uparrow}(s)=
\sqrt{\frac{C_\sigma q^{2s}}{-k^2_q \sin \theta
\sin (\theta +\racion i 2 \log q)}}e^{-\half \partial_s}
- \sqrt{\frac{C_\sigma q^{-2s}}{-k^2_q \sin \theta
\sin (\theta -\racion i 2 \log q)}}
e^{\half \partial_s}, \end{array}
$$
are such that
$$
a^{\uparrow}(s)a^{\downarrow}(s)=\g H_q(s)\quad \mbox{ and }
\quad [a^{\downarrow}(s),a^{\uparrow}(s)]_{1/q}=
\frac{4C_\sigma}{k_q}.
$$

Another possible choice is $A(s)=B(s)=1$ \cite{asuslov5}, hence a
straightforward calculation shows that the two conditions in Theorem
\ref{el-teo-q-lin} hold if $\varsigma = q^{-1}$, thus $\Lambda =
{4C_\sigma}{k_q^{-1}}$. With this choice the orthogonality of the
functions $\Phi_n$ is $\int_{-1}^1\Phi_n(s)\Phi_m(s)dx=
\delta_{n,m}$. In this case, the Hamiltonian is
\begin{small}
$$
\g H_q(s)=
\frac{C_\sigma q}{k_q^2}\left\{\frac{e^{-\partial_s}}{\sin
\theta \sin (\theta \!+\!\racion i 2 \ln q)}+
\frac{e^{\partial_s}}{\sin (\theta\! -\!\racion i 2 \ln q)
\sin \theta}\!-\!\frac{4}{\sqrt{q}}\left(1\!-\!\frac{1+q}
{q+q^{-1}-2\cos2\theta}\right) I \right\},
$$
\end{small}
and
\begin{equation} \label{op_ex_9}
\begin{array}{l}
\dst \g a^{\downarrow}(s) \equiv \g
a^{\downarrow}_{1/2}(s)=\frac{\sqrt{-C_\sigma}}{k_q \sin
\theta}\left(\ e^{\half \partial_s}q^s -
 e^{-\half \partial_s}q^{-s}\right), \\[0.6cm]
\dst \g a^{\uparrow}(s) \equiv \g
a^{\uparrow}_{1/2}(s)=\frac{\sqrt{-C_\sigma}}{k_q \sin
\theta}\left( q^s  e^{-\half \partial_s}  - q^{-s}  e^{\half
\partial_s} \right).
\end{array}
\end{equation}
With these operators
$$
\g H_q(s)=a^{\uparrow}(s) a^{\downarrow}(s) \quad \mbox{ and } \quad
[a^{\downarrow}(s),a^{\uparrow}(s)]_{q^{-1}}=\frac{4C_\sigma}{k_q}.
$$
This case was first considered in \cite{asuslov5}.

\subsection*{Acknowledgments}
We are grateful to A. Ruffing for inviting us to participate at the
Bexbach Conference 2002 and present there our results \cite{ran03}.
The research of RAN has been partially supported by the Ministerio
de Ciencias y Tecnolog\'{\i}a of Spain under the grant
BFM-2003-06335-C01, the Junta de Andaluc\'{\i}a under grant FQM-262. The
participation of NMA in this work has been supported in part by the
UNAM--DGAPA project IN112300.



\begin{thebibliography}{99}

\bibitem{ran}  R. \'{A}lvarez-Nodarse, {\it Polinomios hipergem\'etricos
y $q$-polinomios}. Monograf\'{\i}as del Seminario Matem\'atico
``Garc\'{\i}a de Galdeano'' Vol. {\bf 26}.
Prensas Universitarias de Zaragoza, Zaragoza, Spain, 2003.

\bibitem{ran03} R. \'{A}lvarez-Nodarse, N. M. Atakishiyev, and
R. S. Costas-Santos, Factorization of hypergeometic-type difference
equations on nonuniform lattices: algebraic properties of the
raising and lowering operators (in preparation).

\bibitem{ran-rob} R. \'{A}lvarez-Nodarse and R.S. Costas-Santos,
Factorization method for difference equations of
hypergeometric-type on nonuniform lattices.
{\it J. Phys. A: Math. Gen.} {\bf 34} (2001), 5551-5569.

\bibitem{asuslov1} R. Askey and S. K. Suslov, {The {\it q-}harmonic
oscillator and an analogue of the Charlier polynomials.}
{\it J. Phys. A.} {\bf 26}  (1993), L693-L698.

\bibitem{asuslov2} R. Askey and S. K. Suslov, {The {\it q-}harmonic
oscillator and the Al-Salam and Carlitz polynomials.}
{\it Lett. Math. Phys.} {\bf 29} (1993), 123--132.

\bibitem{ata84} N. M. Atakishiyev, Construction of the dynamical
symmetry group of the relativistic harmonic oscillator by the Infeld-Hull
factorization method. In {\it Group Theoretical Methods in Physics},
M. Serdaroglu and E. Inonu. (Eds.). Lecture Notes in Physics, Springer-Verlag,
{\bf 180}, 1983, 393-396; Ibid. {\it Theor. Math. Phys.} {\bf 56} (1984), 563-572.

\bibitem{ataw94}  N. M. Atakishiyev, A. Frank, and K. B. Wolf,
{A simple difference realization  of the Heisenberg $q$-algebra.}
{\it J. Math. Phys.} {\bf 35} (1994), 3253-3260.

\bibitem{ata98} N. M. Atakishiyev, E.I. Jafarov, S.M. Nagiyev,
and K.B. Wolf, {Meixner oscillators.} {\it Revista Mexicana de
F\'{\i}sica} {\bf 44}(3) (1998), 235-244.

\bibitem{arsus} N. M. Atakishiyev, M. Rahman, and S. K. Suslov,
{On Classical Orthogonal Polynomials.} {\em Constructive Approximation.}
{\bf 11} (1995), 181-226.

\bibitem{asuslov4}  N. M. Atakishiyev and S. K. Suslov,
{Difference  analogs of the harmonic oscillator.}
{\it Theor. Math. Phys.} {\bf 85} (1991), 442-444.

\bibitem{asuslov5}  N. M. Atakishiyev and S. K. Suslov,
{  A realization  of the  {\it q-}harmonic oscillator.}
{\it  Theor. Math. Phys.} {\bf 87} (1991), 1055-1062.

\bibitem{ata94} N. M. Atakishiyev and B. Wolf, Approximation
on a finite set of points through Kravchuk functions. {\it Revista
Mexicana de F\'{\i}sica} {\bf 40} (1994), 1055-1062.

\bibitem{ban98} G. Bangerezako, Discrete Darboux transformation
for discrete polynomials of hypergeometric type. {\it J. Phys A:
Math. Gen.} {\bf 31} (1998), 2191-2196.

\bibitem{ban99}  G. Bangerezako, {The factorization method for
the Askey--Wilson polynomials.} {\it J. Comput. Appl. Math.}
{\bf 107} (1999), 219-232.

\bibitem{ruf1} C. Berg and A. Ruffing, Generalized $q$-Hermite polynomials.
\textit{Comm. Math. Phys.} \textbf{223}  (2001),~29--46.

\bibitem{bie89} L. C. Biedenharn,
{The quantum group $SU_q(2)$ and a {\it q-}analogue of the
boson operators.} {\it J. Phys. A.} {\bf  22} (1989), L873-L878.

\bibitem{vin1}  R. Floreanini, J. LeTourneux, and L. Vinet,
More on the $q$-oscillator algebra and $q$-orthogonal polynomials.
\textit{J. Phys. A: Math. Gen.} {\bf 28} (1995), L287--L293.

\bibitem{vin2} R. Floreanini and L. Vinet, $q$-orthogonal
polynomials and the oscillator quantum group. \textit{Lett. Math. Phys.}
{\bf 22}  (1991), 45--54.

\bibitem{IH} L. Infeld and T. E. Hull, The factorization method.
{\it Review of Modern Physics} {\bf 23} (1951),~21--68.

\bibitem{ks} R. Koekoek and R. F. Swarttouw, {The Askey-scheme of
hypergeometric orthogonal polynomials and its q-analogue.}
{\em Reports of the Faculty of Technical Mathematics and Informatics}
{\bf No. 98-17}. Delft University of Technology, Delft, 1998.


\bibitem{ruf2} A. Lorek and A. Ruffing, and J. A. Wess,
A $q$-deformation of the harmonic oscillator.  \textit{Z. Phys. C}
\textbf{74} (1997),  369--377.

\bibitem{lor01} M. Lorente, {Raising and lowering operators,
factorization method and differential/difference operators of
hypergeometric type.} {\it J. Phys. A: Math. Gen.} {\bf 34}
(2001) 569-588.

\bibitem{mil69} W. Miller (Jr.),
Lie theory and difference equations. I.
{\it J. Math. Anal. Appl.} {\bf  28} (1969), 383--399.

\bibitem{mil70} W. Miller (Jr.),
Lie theory and $q$-difference equations.
{\it SIAM J. Math. Anal.} {\bf  1} (1970), 171--188.

\bibitem{nag95} Sh. M. Nagiyev, Difference Schr\"odinger equation
and $q$-oscillator model, {\it Theor. Math. Phys.} {\bf 102}
(1995), 180-187.

\bibitem{nsu}  A. F. \ Nikiforov, S. K.\ Suslov, and V. B. \ Uvarov,
{\em Classical Orthogonal Polynomials of a Discrete Variable.}
Springer Series in Computational Physics. Springer-Verlag,
Berlin, 1991.

\bibitem{nubook} A. F.Nikiforov and V.B.Uvarov, The Special
Functions of Mathematical Physics, Birkh\"{a}user, Basel, 1988.

\bibitem{smi99} Yu. F. Smirnov, {Factorization method: New aspects.}
{\it Revista Mexicana de F\'{\i}sica} {\bf 45} (1999),~1-6.

\end{thebibliography}
\end{document}